\documentclass[12pt, a4paper]{article}

\usepackage{amsmath, amssymb, amsthm}
\usepackage{mathtools}

\usepackage[top=25mm, bottom=25mm, left=30mm, right=30mm]{geometry}
\usepackage{setspace}
\usepackage{booktabs}
\usepackage{array}
\usepackage{longtable}
\usepackage{multirow}

\usepackage{algorithm}
\usepackage{algpseudocode}

\usepackage{graphicx}

\usepackage{xcolor}
\usepackage[colorlinks=true, linkcolor=blue!60!black, citecolor=blue!50!black, urlcolor=blue!60!black]{hyperref}

\newtheorem{theorem}{Theorem}[section]
\newtheorem{lemma}[theorem]{Lemma}
\newtheorem{proposition}[theorem]{Proposition}
\newtheorem{corollary}[theorem]{Corollary}
\newtheorem{definition}{Definition}[section]
\newtheorem{remark}{Remark}[section]
\newtheorem{example}{Example}[section]

\begin{document}

\title{Counting Weight-\(k\) Italian Dominating Sets on Trees and Cycles}
\author{%
  Pingping Shao\thanks{College of Science, China Jiliang University, Hangzhou 310018, China} \ \ and Chengye Zhao\thanks{Corresponding author. College of Science, China Jiliang University, Hangzhou 310018, China; E-mail: \texttt{07A0802084@cjlu.edu.cn}}
}
\date{}

\maketitle

\begin{abstract}
We study the problem of counting \emph{Italian dominating functions} (IDFs) of each possible weight on paths, trees, and cycles. An Italian dominating function on a graph $G=(V,E)$ is a function $f\colon V\to\{0,1,2\}$ such that every vertex $v$ with $f(v)=0$ satisfies $\sum_{u\in N(v)}f(u)\ge 2$; the \emph{Italian domination polynomial} $D_I(G,x)=\sum_k d_I(G,k)\,x^k$ records the number $d_I(G,k)$ of such functions of weight~$k$. This polynomial is the Italian analogue of the well-studied domination and Roman domination polynomials, and computing its coefficients on a fixed graph class is, in essence, a \emph{constrained integer composition} problem: passing to the deficit variables $y(v)=2-f(v)$, each coefficient $d_I(G,2n-j)$ counts the number of compositions of $j$ into $n$ parts, each at most~$2$, that additionally satisfy a local adjacency capacity constraint. We make this connection explicit and use it to derive closed-form expressions for $d_I(C_n,2n-j)$, $j=0,\ldots,5$, as trinomial coefficients minus explicitly enumerated forbidden configurations. On the algorithmic side, we give dynamic programming algorithms that compute \emph{all} coefficients of $D_I$ on path graphs ($O(n^2)$ time, $O(n)$ space), on general trees ($O(n^2)$ time, $O(n)$ space, via a post-order merge with a deficit parameter), and on cycle graphs ($O(n^2)$ time, $O(n)$ space, via boundary conditioning), together with formal correctness proofs and complexity analyses. The known values $\gamma_I(P_n)=\lfloor n/2\rfloor+1$ and $\gamma_I(C_n)=\lceil n/2\rceil$ are recalled with references rather than reproved. All numerical results are verified against brute-force enumeration, and complete tables are provided.
\end{abstract}


\bigskip

\section{Introduction}

Graph domination is a fundamental topic in combinatorial optimization with applications in network security, facility location, and resource allocation. The concept originated from the Roman Empire defense strategy studied by Cockayne et al.\ \cite{cockayne2004}, where each region is assigned $0$, $1$, or $2$ legions, and every undefended region must have a neighbor with two legions.

Let $G=(V,E)$ be a graph. A \emph{Roman dominating function} (RDF) is a function $f\colon V\to\{0,1,2\}$ such that every vertex $v$ with $f(v)=0$ has at least one neighbor $u$ with $f(u)=2$. The minimum weight $\omega(f)=\sum_{v\in V} f(v)$ over all RDFs is the \emph{Roman domination number} $\gamma_R(G)$.

Chellali et al.\ \cite{chellali2016} introduced \emph{Roman~$\{2\}$-domination} (also called \emph{Italian domination}), relaxing the condition: a function $f\colon V\to\{0,1,2\}$ is an \emph{Italian dominating function} (IDF) if every vertex $v$ with $f(v)=0$ satisfies $\sum_{u\in N(v)} f(u)\ge 2$. This allows two neighbors each assigned~$1$ to jointly protect an undefended vertex, rather than requiring a single neighbor assigned~$2$. The minimum weight over all Italian dominating functions is the \emph{Italian domination number} $\gamma_I(G)$. The Italian domination number was first studied under the name Roman~$\{2\}$-domination \cite{chellali2016}; the parameter and its variants are surveyed in \cite{chellali2021}, and several of its basic values were obtained in the manuscripts and papers \cite{brouwer2012,henningklostermeyer2017,klostermeyer2019}. We refer to Haynes, Hedetniemi and Henning \cite{haynes2020} for general background on domination in graphs.

Beyond finding the minimum weight, a natural and more refined problem is to \emph{count} the number of Italian dominating functions of each possible weight. This leads to the Italian domination polynomial:

\begin{definition}[Italian domination polynomial]
\label{def:poly}
The \emph{Italian domination polynomial} of a graph $G$ is
\[
  D_I(G,x)=\sum_{k=\gamma_I(G)}^{2|V(G)|} d_I(G,k)\,x^k,
\]
where $d_I(G,k)$ counts the Italian dominating functions of weight~$k$.
\end{definition}

We refer to the set of Italian dominating functions of weight~$k$ as the \emph{weight-$k$ Italian dominating set}, and $d_I(G,k)$ counts its cardinality.

\begin{definition}[weight-$k$ Italian dominating set]
\label{def:korder}
Given a graph $G=(V,E)$ and a positive integer $k$, the \emph{weight-$k$ Italian dominating set} of~$G$ is the collection of all Italian dominating functions $f\colon V\to\{0,1,2\}$ on~$G$ with total weight $\omega(f)=\sum_{v\in V}f(v)=k$. Its cardinality is $d_I(G,k)$.
\end{definition}

\subsection{Related work}

The polynomial that counts the number of dominating sets of each size, the \emph{domination polynomial} $D(G,x)$, was introduced by Alikhani and Peng \cite{alikhani2009} and has been studied extensively; in particular, Kotek et al.\ \cite{kotek2012} proved recurrence relations and splitting formulas for $D(G,x)$ based on articulation vertices, and Kotek, Preen and Tittmann \cite{kotek2014} gave subset-sum representations of domination polynomials. The same counting philosophy has been applied to Roman domination: Mirhoseini and Jafari Rad \cite{mirhoseini2024} introduced the \emph{Roman domination polynomial} and computed it for several graph classes, and Deepak et al.\ \cite{deepak2021} determined the Roman domination polynomial of cycles; we compare our cycle closed forms with theirs in Remark~\ref{rem:romancompare}. On the counting side, Zheng and Koh \cite{zheng2014} gave exact algorithms for counting minimum Roman dominating functions, and Abu-Khzam, Fernau and Mann \cite{abukhzam2024} studied the enumeration of (minimal) Roman dominating functions on graph classes such as paths and forests. On the complexity side, counting dominating sets is $\#$P-complete even for split and chordal bipartite graphs \cite{kijima2011}. Closest in spirit to the present work are the linear-time dynamic program of Wei and Lu \cite{weilu2019} for an independent variant of Italian domination on block graphs (a minimum-weight problem, not a counting one) and the transfer-matrix algorithm of Mertens \cite{mertens2024} for computing domination polynomials of grids, cylinders, tori, and king graphs. Neither of these works treats the Italian domination polynomial studied here.

The Italian domination polynomial $D_I(G,x)$ studied here is the direct Italian analogue of the Roman domination polynomial of \cite{mirhoseini2024,deepak2021}; to the best of our knowledge its systematic algorithmic study on paths, trees, and cycles has not appeared before. A central observation of this paper is that the counting problem behind $D_I$ is naturally a \emph{constrained integer composition} problem, connecting it to the classical theory of integer partitions \cite{andrews1998}: via the deficit transform $y(v)=2-f(v)$, the coefficient $d_I(G,2n-j)$ counts compositions of $j$ into $n$ parts (each in $\{0,1,2\}$) that satisfy a local adjacency capacity constraint. Compositions whose parts are constrained by local adjacency conditions form a classical subject of enumerative combinatorics, studied asymptotically by Bender and Canfield \cite{bendercanfield2009}; the aim here is not to add new enumeration results within that theory but to establish a bridge from a graph-theoretic counting problem to this classical setting and to exploit the resulting structure algorithmically. This composition viewpoint provides both clean closed-form expressions (Section~\ref{sec:partitions} and Section~\ref{sec:closedform}) and a transparent interpretation of the dynamic programs in Sections~\ref{sec:path}--\ref{sec:cycle}.

\subsection{Our contributions}

Computing the full coefficient vector $(d_I(G,k))_{k}$ is significantly harder than finding $\gamma_I(G)$, since one must enumerate all valid functions, not just find the minimum. For general graphs, the problem is believed to be computationally hard: the $\{0,1\}$-valued case already corresponds to counting subsets $D\subseteq V$ such that every vertex outside $D$ has at least two neighbors in $D$, a $2$-domination-type counting problem. Hardness results for closely related problems corroborate this: counting dominating sets is $\#$P-complete even for split and chordal bipartite graphs \cite{kijima2011}, and deciding whether a graph admits an Italian dominating function of weight at most a given bound is NP-complete \cite{ferrari2025}. For structured graph classes such as paths, trees, and cycles, however, efficient algorithms exist. The contributions of this paper are:

\begin{enumerate}
  \item We establish the connection between counting Italian dominating functions and \emph{constrained integer compositions}: with the deficit variables $y(v)=2-f(v)$, the coefficient $d_I(G,2n-j)$ equals the number of compositions of $j$ into $n$ parts, each at most~$2$, that satisfy the local capacity constraint $\sum_{u\in N(v)}y(u)\le 2\deg(v)-2$ for every $v$ with $y(v)=2$ (Lemma~\ref{lem:deficit}). The unconstrained baseline is the trinomial coefficient $T(n,j)=[z^j](1+z+z^2)^n$ (Section~\ref{sec:partitions}).
  \item We present a chain dynamic programming algorithm for computing all coefficients of $D_I(P_n, x)$ on path graphs in $O(n^2)$ time and $O(n)$ space, with a correctness proof and a complexity analysis (Section~\ref{sec:path}).
  \item We develop a tree DP algorithm for computing all coefficients of $D_I(T, x)$ on any tree~$T$ with $n$ vertices in $O(n^2)$ time and $O(n)$ space, based on a post-order merge with a \emph{deficit} parameter, again with a correctness proof and a complexity analysis (Section~\ref{sec:tree}).
  \item We develop a cycle DP algorithm for computing all coefficients of $D_I(C_n, x)$ in $O(n^2)$ time and $O(n)$ space by \emph{boundary conditioning}, i.e., by conditioning on the values of two adjacent vertices and reducing the cycle to a modified path DP (Section~\ref{sec:cycle}).
  \item We recall the known values $\gamma_I(P_n)=\lfloor n/2\rfloor+1$ (Proposition~\ref{prop:gamma_path}) and $\gamma_I(C_n)=\lceil n/2\rceil$ (Proposition~\ref{prop:gamma_cycle}) with appropriate references, and derive closed-form expressions for $d_I(C_n, 2n-j)$, $j=0,\ldots,5$, presented as trinomial coefficients minus explicitly enumerated forbidden configurations, and we show that $d_I(C_n,2n-j)$ is a polynomial in $n$ of degree $j$ for each fixed $j$ (Section~\ref{sec:partitions}, Lemma~\ref{lem:poly}).
  \item We establish the coefficient-wise dominance $D_I(C_n,x)\ge D_I(P_n,x)$ as a direct consequence of the deficit formulation (Remark~\ref{rem:dom}).
  \item All algorithms are verified against brute-force enumeration, and complete numerical tables are provided for path graphs $P_n$ ($n=1,\ldots,10$), star graphs $K_{1,m}$ ($m=1,\ldots,8$), cycle graphs $C_n$ ($n=3,\ldots,12$), and several additional trees (Section~\ref{sec:numerics}).
\end{enumerate}

\section{Preliminaries}

\subsection{Notation}

Let $G=(V,E)$ be a simple, finite, undirected graph with vertex set $V$ and edge set~$E$. The \emph{open neighborhood} of a vertex $v\in V$ is $N(v)=\{u\in V : uv\in E\}$, and the \emph{closed neighborhood} is $N[v]=N(v)\cup\{v\}$. The \emph{degree} of $v$ is $\deg(v)=|N(v)|$.

A \emph{path graph} $P_n$ has vertices $v_1, v_2, \ldots, v_n$ and edges $v_i v_{i+1}$ for $i=1,\ldots,n-1$. A \emph{cycle graph} $C_n$ ($n\ge 3$) has vertices $v_1, v_2, \ldots, v_n$ and edges $v_i v_{i+1}$ for $i=1,\ldots,n-1$ plus the edge $v_n v_1$. A \emph{star graph} $K_{1,m}$ has a center vertex connected to $m$ leaves. A \emph{tree} $T=(V,E)$ is a connected acyclic graph. Any tree can be rooted at an arbitrary vertex~$r$, inducing a parent--child relationship: for each vertex $v\ne r$, its \emph{parent} is the neighbor on the unique path from $v$ to~$r$, and its \emph{children} are the remaining neighbors. A vertex with no children is a \emph{leaf}.

Note that $C_n$ can be obtained from $P_n$ by adding the \emph{closing edge} $v_n v_1$. This observation is central to our cycle algorithm (Section~\ref{sec:cycle}) and to the path--cycle comparison of Remark~\ref{rem:dom}.

\subsection{Italian domination}

\begin{definition}[Italian dominating function]
\label{def:idf}
An \emph{Italian dominating function} (IDF) on $G=(V,E)$ is a function $f\colon V\to\{0,1,2\}$ such that for every $v\in V$ with $f(v)=0$,
\[
  \sum_{u\in N(v)}f(u)\ge 2.
\]
The \emph{weight} of $f$ is $\omega(f)=\sum_{v\in V}f(v)$. The minimum weight over all Italian dominating functions is the \emph{Italian domination number}, denoted $\gamma_I(G)$.
\end{definition}

\begin{remark}
Italian domination coincides with the Roman~$\{2\}$-domination of Chellali et al.\ \cite{chellali2016}. It differs from Roman domination (Cockayne et al.\ \cite{cockayne2004}) in that Roman domination requires a neighbor with $f(u)=2$, while Italian domination allows two neighbors with $f(u)=1$ to jointly satisfy the constraint. We refer to \cite{chellali2021} for a survey of Roman-type domination variants.
\end{remark}

\begin{proposition}
\label{prop:basic}
Let $G$ be a graph with $n$ vertices. Then:
\begin{enumerate}
\item[(i)] If $v$ is an isolated vertex (i.e., $\deg(v)=0$), then $f(v)\ge 1$ for any IDF~$f$.
\item[(ii)] $\gamma_I(G)\ge 1$ for any non-empty graph.
\item[(iii)] $d_I(G, 2n)=1$, corresponding to the unique function $f\equiv 2$.
\end{enumerate}
\end{proposition}

\begin{proof}
(i) If $f(v)=0$ and $\deg(v)=0$, then $\sum_{u\in N(v)}f(u)=0<2$, a contradiction. (ii) follows from~(i). (iii) The function $f(v)=2$ for all $v$ has weight $2n$ and is the unique function achieving this weight, since any vertex with $f(v)<2$ would reduce the total below~$2n$.
\end{proof}

\subsection{Known values of the Italian domination number}

The following values are known; we state them for completeness and do not reproduce their proofs.

\begin{proposition}
\label{prop:gamma_path}
For the path graph $P_n$ ($n\ge 1$),
\[
  \gamma_I(P_n) = \left\lfloor \frac{n}{2}\right\rfloor + 1.
\]
\end{proposition}

This is a known result: Klostermeyer and MacGillivray \cite{klostermeyer2019} proved $\gamma_I(P_n)=\lceil (n+1)/2\rceil=\lfloor n/2\rfloor+1$, and the value also follows from the framework of Chellali et al.\ \cite{chellali2016}; see also \cite{henningklostermeyer2017} for the tree case.

\begin{proposition}
\label{prop:gamma_cycle}
For the cycle graph $C_n$ ($n\ge 3$),
\[
  \gamma_I(C_n) = \left\lceil \frac{n}{2} \right\rceil.
\]
\end{proposition}

This result is also known and is recorded, for instance, by Ferrari, Leoni and Lopez Pujato \cite{ferrari2025} (and the references therein). Our cycle algorithm of Section~\ref{sec:cycle} provides an independent computational confirmation of the value for all $n\le 12$ (Table~\ref{tab:cycle_a}).

\section{Counting, Deficit Variables, and Constrained Integer Compositions}
\label{sec:partitions}

In this section we expose the structural connection between the counting problem for Italian dominating functions and a classical problem in additive combinatorics: counting \emph{integer compositions} (ordered partitions into nonnegative parts) subject to constraints. This connection is used in Section~\ref{sec:closedform} to derive closed forms and gives a transparent interpretation of the dynamic programs of Sections~\ref{sec:path}--\ref{sec:cycle}.

\subsection{Deficit variables}

\begin{definition}[Deficit variables]
\label{def:deficit}
Let $G=(V,E)$ have $n$ vertices and let $f\colon V\to\{0,1,2\}$. The \emph{deficit} of $f$ is the vector $y\colon V\to\{0,1,2\}$ defined by
\[
  y(v) \;=\; 2 - f(v) \qquad (v\in V).
\]
Clearly $f(v)=0$ iff $y(v)=2$, and
\[
  \sum_{v\in V} y(v) \;=\; 2n - \omega(f).
\]
\end{definition}

\begin{lemma}[Deficit reformulation]
\label{lem:deficit}
Let $G$ be a graph with $n$ vertices, let $f\colon V\to\{0,1,2\}$, and let $y(v)=2-f(v)$. Then $f$ is an Italian dominating function of weight $2n-j$ if and only if
\begin{enumerate}
\item[(P1)] $\displaystyle\sum_{v\in V} y(v) = j$, and
\item[(P2)] for every vertex $v$ with $y(v)=2$,
\begin{equation}
  \sum_{u\in N(v)} y(u) \;\le\; 2\deg(v) - 2 . \label{eq:p2}
\end{equation}
\end{enumerate}
\end{lemma}

\begin{proof}
By Definition~\ref{def:deficit}, $\omega(f)=2n-j$ iff (P1) holds, and $f(v)=0$ iff $y(v)=2$. For such a vertex $v$, the Italian condition $\sum_{u\in N(v)} f(u)\ge 2$ is, substituting $f(u)=2-y(u)$,
\[
  2\deg(v) - \sum_{u\in N(v)} y(u) \;\ge\; 2
  \quad\Longleftrightarrow\quad
  \sum_{u\in N(v)} y(u) \;\le\; 2\deg(v)-2 ,
\]
which is \eqref{eq:p2}. Vertices with $y(v)\ne 2$ impose no condition. The equivalence follows.
\end{proof}

\begin{corollary}[Cycles]
\label{cor:cycle_deficit}
For the cycle $C_n$ ($n\ge 3$), a vector $y\in\{0,1,2\}^n$ corresponds to an IDF of weight $2n-j$ if and only if $\sum_{i=1}^n y_i=j$ and
\[
  y_i = 2 \;\Longrightarrow\; y_{i-1} + y_{i+1} \le 2 \qquad (\text{indices modulo } n).
\]
\end{corollary}

\begin{proof}
Every vertex of $C_n$ has degree $2$, so $2\deg(v)-2=2$ for all $v$. The claim follows from Lemma~\ref{lem:deficit}.
\end{proof}

\begin{remark}[Endpoint constraints for paths]
\label{rem:endpoints}
The endpoints of the path $P_n$ have degree $1$. Hence, by (P2), any deficit vector $y$ of an IDF on $P_n$ satisfies
\[
  y_1=2 \;\Longrightarrow\; y_2=0, \qquad y_n=2 \;\Longrightarrow\; y_{n-1}=0.
\]
These endpoint constraints are the \emph{only} structural difference between paths and cycles in the deficit formulation, since all internal vertices of $P_n$ and all vertices of $C_n$ have degree $2$. This observation explains both the coefficient-wise dominance $d_I(C_n,k)\ge d_I(P_n,k)$ (Remark~\ref{rem:dom}) and why the closed forms of Section~\ref{sec:closedform} are cleaner for cycles than for paths.
\end{remark}

\subsection{The unconstrained count: trinomial coefficients}

Ignoring the local capacity constraints (P2), the objects counted by (P1) are precisely the compositions of $j$ into $n$ parts, each part in $\{0,1,2\}$.

\begin{definition}[Compositions with parts at most two]
\label{def:trinom}
Let $n\ge 1$ and $0\le j\le 2n$. We denote by
\[
  T(n,j) \;=\; \bigl[z^j\bigr]\,(1+z+z^2)^n
  \;=\; \sum_{t=0}^{\lfloor j/2\rfloor} \binom{n}{t}\binom{n-t}{j-2t}
\]
the number of compositions of $j$ into $n$ parts, each part in $\{0,1,2\}$: choose $t$ parts equal to $2$ and $j-2t$ parts equal to $1$ among the remaining positions. $T(n,j)$ is a \emph{trinomial coefficient} (the coefficient of $z^j$ in $(1+z+z^2)^n$).
\end{definition}

\begin{remark}[Relation to classical partition theory]
\label{rem:fibonacci}
The number $T(n,j)$ belongs to the classical theory of restricted compositions and partitions \cite{andrews1998}: compositions of $j$ whose parts belong to $\{1,2\}$ (with an unrestricted number of parts) are counted by the Fibonacci number $F_{j+1}$, and $T(n,j)$ is the variant where the number of parts is fixed to exactly $n$ (zeros allowed). Summing over all $j$ gives $\sum_{j=0}^{2n} T(n,j)=3^n$, and $T(n,j)=T(n,2n-j)$ by the substitution $y\mapsto 2-y$ (equivalently $f\mapsto 2-f$). The Italian counting problem $d_I(G,2n-j)$ is thus the problem of counting \emph{admissible} compositions of this family, where the local capacity constraints (P2) prune a subset of the configurations. For fixed $j$, the trinomial coefficient $T(n,j)$ is a polynomial in $n$ of degree $j$ with leading coefficient $1/j!$.
\end{remark}

\begin{proposition}[Baseline and small-\(j\) exactness]
\label{prop:baseline}
Let $G$ be a graph with $n$ vertices. Then for every $j\ge 0$,
\[
  d_I(G, 2n-j) \;\le\; T(n,j) .
\]
Moreover, if $G$ has minimum degree at least $1$ (in particular, for $P_n$ with $n\ge 2$, for $C_n$, and for every tree with at least two vertices), then equality holds for $j=0,1,2$:
\[
  d_I(G,2n)=1, \qquad d_I(G,2n-1)=n, \qquad d_I(G,2n-2)=\binom{n}{2}+n .
\]
\end{proposition}

\begin{proof}
By Lemma~\ref{lem:deficit}, $d_I(G,2n-j)\le T(n,j)$, with equality iff no configuration counted by $T(n,j)$ violates (P2). For $j=0$ the only configuration is $y\equiv 0$, which satisfies (P2) vacuously. For $j=1$, the only configurations have a single $2$-entry; the neighbours of that entry contribute $0$ to the left-hand side of (P2), and $2\deg(v)-2\ge 0$ whenever $\deg(v)\ge 1$, so no violation occurs. For $j=2$, configurations have either a single $2$-entry or two $1$-entries; the former is handled as in the case $j=1$, and the latter imposes no condition. Hence equality holds for $j\le 2$ on every graph of minimum degree at least~$1$. (The isolated-vertex case $P_1$ is the single exception: $y(v)=2$ is excluded because $2\deg(v)-2=-2$; then $d_I(P_1,0)=0<1=T(1,2)$.)
\end{proof}

\begin{remark}
The deficit formulation also gives the counting identity
\begin{equation}
  d_I(C_n, 2n-j) \;=\; T(n,j) - E(n,j), \label{eq:tme}
\end{equation}
where $E(n,j)$ counts the compositions counted by $T(n,j)$ that violate the cycle constraint of Corollary~\ref{cor:cycle_deficit}. For $j\le 4$ no configuration is excluded ($E(n,j)=0$); for $j=5$ one finds
\begin{equation}
  E(n,5) \;=\; 2n \qquad (n\ge 4), \label{eq:e5}
\end{equation}
as shown in Section~\ref{sec:closedform}.
\end{remark}

\begin{lemma}[Polynomial structure of the excluded configurations]
\label{lem:poly}
Fix $j\ge 0$. Then $E(n,j)=0$ for $j\le 4$, and for every $j\ge 5$ the number $E(n,j)$ is a polynomial in $n$ of degree at most $j-4$ for all sufficiently large $n$. Consequently, for each fixed $j$ the coefficient $d_I(C_n,2n-j)=T(n,j)-E(n,j)$ is a polynomial in $n$ of degree $j$ with leading coefficient $1/j!$.
\end{lemma}

\begin{proof}
A configuration counted by $E(n,j)$ contains at least one \emph{forbidden pattern}, i.e., a triple $(y_{i-1},y_i,y_{i+1})$ with $y_i=2$ and $y_{i-1}+y_{i+1}\ge 3$; the only such triples are $(2,2,1)$, $(1,2,2)$, and $(2,2,2)$ (indices modulo $n$), of total weight $5$, $5$, and $6$, respectively. In particular, $E(n,j)=0$ for $j\le 4$. For $j\ge 5$, decompose every excluded configuration into its maximal contiguous \emph{bad blocks}: a bad block is a maximal run of positions covered by forbidden patterns, together with its internal pattern structure. Each bad block has total weight at least $5$ and length at most $3\lfloor j/5\rfloor$, and for fixed $j$ there are finitely many possible block shapes; distinct bad blocks are separated by at least one position outside every forbidden pattern. An excluded configuration of total weight $j$ consists of $t\ge 1$ bad blocks of total weight $w\ge 5t$, plus $s=j-w\le j-5t$ further positions of weight $1$ that are not adjacent to any bad block (a position of weight $1$ adjacent to a block would merge into it by maximality); positions of weight $0$ are unrestricted. Fixing the $t$ block shapes, their cyclic order, and the separation constraints, the number of ways to place the blocks and the $s$ unit positions on the $n$ positions of the cycle is a polynomial in $n$ of degree $t+s\le j-4t\le j-4$: one chooses the starting position of the first block ($n$ choices) and then the positions of the remaining blocks and units subject to fixed separation constraints, a binomial-type count; the cyclic wrap-around is absorbed by the constant number of ways to cut the cycle. Summing over the finitely many shape vectors and gap patterns gives the claim. The polynomiality of $d_I(C_n,2n-j)$ and its leading coefficient follow from $d_I=T-E$ and Remark~\ref{rem:fibonacci}.
\end{proof}

\begin{remark}
By Lemma~\ref{lem:poly}, $d_I(C_n,2n-j)$ is a polynomial in $n$ of degree $j$ with leading coefficient $1/j!$ for every fixed $j$, a fact that is also confirmed numerically in Section~\ref{sec:numerics}.
\end{remark}

\section{Algorithm for Path Graphs}
\label{sec:path}

In this section, we develop a dynamic programming algorithm to compute all coefficients $d_I(P_n, k)$ for the path graph~$P_n$. From the viewpoint of Section~\ref{sec:partitions}, the algorithm enumerates exactly the admissible compositions of the deficit vectors; we formulate it directly on the original variables $x_i=f(v_i)$.

\subsection{Problem formulation}

Label the vertices of $P_n$ as $v_1, v_2, \ldots, v_n$ in path order. An Italian dominating function $f$ corresponds to a sequence $(x_1, x_2, \ldots, x_n)$ where $x_i = f(v_i) \in \{0, 1, 2\}$, satisfying:
\begin{enumerate}
  \item[(C1)] For $i=1$: if $x_1=0$, then $x_2 \ge 2$.
  \item[(C2)] For $2\le i\le n-1$: if $x_i=0$, then $x_{i-1}+x_{i+1}\ge 2$.
  \item[(C3)] For $i=n$: if $x_n=0$, then $x_{n-1}\ge 2$.
\end{enumerate}

The weight is $k = \sum_{i=1}^n x_i$, and $d_I(P_n, k)$ equals the number of valid sequences of weight~$k$.

\subsection{DP state definition}

The key observation is that when processing vertices left to right, the constraint on $x_i$ depends on $x_{i-1}$ (the previous value), and the constraint that $x_i$ imposes on $x_{i+1}$ depends on $x_i$ and $x_{i-1}$. We capture this with a \emph{propagated lower bound}~$c$:

\begin{definition}[State]
\label{def:pathstate}
The DP state at position $i$ is a triple $(\textit{prev},\, c,\, s)$ where:
\begin{itemize}
  \item $\textit{prev}\in\{0,1,2\}$: the value $x_{i-1}$ assigned to the previous vertex;
  \item $c\in\{0,1,2\}$: the lower-bound constraint on the current vertex $x_i$ (i.e., $x_i\ge c$);
  \item $s\in\{0,1,\ldots,2n\}$: the prefix sum $x_1+x_2+\cdots+x_{i-1}$.
\end{itemize}
The value $\mathrm{dp}[\textit{prev}][c][s]$ stores the number of valid partial assignments achieving this state.
\end{definition}

The constraint $c$ arises from the Italian domination condition: if $x_{i-1}=0$, then $x_{i-2}+x_i\ge 2$, so $x_i\ge 2 - x_{i-2}$. Since $x_{i-2}=\textit{prev}$ of the previous step, this is propagated as $c = \max(0,\, 2 - \textit{prev}_{\text{old}})$.

\subsection{Algorithm}

\begin{algorithm}[htbp]
\caption{Path graph Italian domination polynomial}
\label{alg:path}
\small
\begin{algorithmic}[1]
\Require Path length $n\ge 1$
\Ensure Coefficients $d_I(P_n, k)$ for $k=0,\ldots,2n$
\State $\textit{max\_sum}\gets 2n$
\If{$n=1$}
  \State \Return $\{1\mapsto 1,\; 2\mapsto 1\}$ \Comment{Isolated vertex: $f(v)=0$ is invalid}
\EndIf
\State Initialize $\mathrm{dp}$ as empty map
\For{$x_1\gets 0$ \textbf{to} $2$}
  \If{$x_1=0$}
    \State $c_1\gets 2$ \Comment{Next vertex must be $\ge 2$}
  \Else
    \State $c_1\gets 0$
  \EndIf
  \State $\mathrm{dp}[x_1][c_1][x_1]\gets \mathrm{dp}[x_1][c_1][x_1]+1$
\EndFor
\For{$i\gets 2$ \textbf{to} $n-1$}
  \State Initialize $\mathrm{dp\_new}$ as empty map
  \ForAll{states $(\textit{prev},\, c_{\textit{prev}},\, s)$ with $\mathrm{dp}[\textit{prev}][c_{\textit{prev}}][s]>0$}
    \For{$x_i\gets c_{\textit{prev}}$ \textbf{to} $2$}
      \State $s'\gets s+x_i$
      \If{$s'>\textit{max\_sum}$} \textbf{continue} \EndIf
      \If{$x_i=0$}
        \State $c_{\textit{new}}\gets \max(0,\, 2-\textit{prev})$
      \Else
        \State $c_{\textit{new}}\gets 0$
      \EndIf
      \State $\mathrm{dp\_new}[x_i][c_{\textit{new}}][s']\gets \mathrm{dp\_new}[x_i][c_{\textit{new}}][s']+\mathrm{dp}[\textit{prev}][c_{\textit{prev}}][s]$
    \EndFor
  \EndFor
  \State $\mathrm{dp}\gets \mathrm{dp\_new}$
\EndFor
\State Initialize $\textit{result}[k]\gets 0$ for all $k$
\ForAll{states $(\textit{prev},\, c_{\textit{prev}},\, s)$ with $\mathrm{dp}[\textit{prev}][c_{\textit{prev}}][s]>0$}
  \For{$x_n\gets c_{\textit{prev}}$ \textbf{to} $2$}
    \If{$x_n=0$ \textbf{and} $\textit{prev}<2$} \textbf{continue} \EndIf \Comment{Terminal constraint (C3)}
    \State $k\gets s+x_n$
    \State $\textit{result}[k]\gets \textit{result}[k]+\mathrm{dp}[\textit{prev}][c_{\textit{prev}}][s]$
  \EndFor
\EndFor
\State \Return $\textit{result}$
\end{algorithmic}
\end{algorithm}

\subsection{Correctness}

\begin{theorem}
\label{thm:pathcorrect}
Algorithm~\ref{alg:path} correctly computes $d_I(P_n, k)$ for all $k$.
\end{theorem}

\begin{proof}
We prove by induction on the position $i$ that after processing positions $1,\ldots,i-1$, the value $\mathrm{dp}[\textit{prev}][c][s]$ equals the number of valid partial sequences $(x_1,\ldots,x_{i-1})$ satisfying:
\begin{itemize}
  \item Constraints (C1) for position~$1$ (if $i>1$) and (C2) for positions $2,\ldots,i-1$;
  \item $x_{i-1}=\textit{prev}$;
  \item $x_i\ge c$ is necessary and sufficient for satisfying the Italian domination constraint at position $i-1$ (if $x_{i-1}=0$) or no constraint (if $x_{i-1}>0$);
  \item $s = x_1+\cdots+x_{i-1}$.
\end{itemize}

\emph{Base case} ($i=2$): After initialization, for $x_1=0$ we set $c_1=2$, which is exactly the constraint $x_2\ge 2$ required by~(C1). For $x_1>0$, $c_1=0$ (no constraint on $x_2$). The prefix sum is $s=x_1$. So the invariant holds.

\emph{Inductive step} ($i\to i+1$): Consider a state $(\textit{prev}, c_{\textit{prev}}, s)$ at position~$i$. By the induction hypothesis, $c_{\textit{prev}}$ is the correct lower bound on $x_i$. We enumerate all valid $x_i\ge c_{\textit{prev}}$:
\begin{itemize}
  \item If $x_i=0$: the Italian domination constraint at position $i$ requires $x_{i-1}+x_{i+1}\ge 2$, i.e., $x_{i+1}\ge 2-\textit{prev}$. We set $c_{\textit{new}}=\max(0, 2-\textit{prev})$, which is the correct lower bound for~$x_{i+1}$.
  \item If $x_i>0$: no constraint is imposed on $x_{i+1}$ by position~$i$, so $c_{\textit{new}}=0$.
\end{itemize}
The new prefix sum $s'=s+x_i$ is correctly updated. This establishes the invariant for position~$i+1$.

\emph{Termination} ($i=n$): At the last position, the terminal constraint (C3) requires that if $x_n=0$, then $x_{n-1}\ge 2$, i.e., $\textit{prev}\ge 2$. This is checked in the final loop. The final weight is $k=s+x_n$. Summing over all valid states and choices of $x_n$ gives $d_I(P_n,k)$.

By induction, every valid sequence is counted exactly once and no invalid sequence is counted, proving the theorem.
\end{proof}

\subsection{Complexity}

\begin{theorem}
\label{thm:pathcomplex}
Algorithm~\ref{alg:path} runs in $O(n^2)$ time and $O(n)$ space.
\end{theorem}

\begin{proof}
At each position $i$, the reachable states have $3$ choices for $\textit{prev}$, $3$ choices for~$c$, and at most $2(i-1)+1\le 2n+1$ choices for~$s$, giving $O(n)$ states per position. For each state, the inner loop enumerates at most $3$ values of~$x_i$, each transition costing $O(1)$. With $n-2$ middle positions plus initialization and termination, the total time is $O(n \cdot 3 \cdot n)=O(n^2)$.

Space: only two maps ($\mathrm{dp}$ and $\mathrm{dp\_new}$), each with $O(n)$ states, are kept simultaneously; the result array has $O(n)$ entries. Hence the space usage is $O(n)$.
\end{proof}

\begin{remark}
In deficit variables, the chain DP counts exactly the admissible compositions of Section~\ref{sec:partitions}: the propagated bound $c$ implements the endpoint constraint $y_1=2\Rightarrow y_2=0$ of Remark~\ref{rem:endpoints} at position $1$ (via $x_1=0\Rightarrow x_2\ge 2$, i.e., $y_2\le 0$), and the internal constraints implement (P2) for degree-$2$ vertices.
\end{remark}

\section{Algorithm for Trees}
\label{sec:tree}

We now extend the approach to general trees using a tree DP with a \emph{deficit} parameter.

\subsection{Overview}

Root the tree $T$ at an arbitrary vertex~$r$. Process vertices in post-order (children before parent). For each vertex~$v$, we compute a DP over its subtree that tracks:
\begin{itemize}
  \item $f_v\in\{0,1,2\}$: the value assigned to~$v$;
  \item $\delta\in\{0,1,2\}$: the \emph{deficit}, i.e., how much $v$ still needs from its parent to satisfy its own Italian domination constraint;
  \item $w$: the total weight in the subtree of~$v$.
\end{itemize}

The deficit is only nonzero when $f_v=0$. In that case, $\delta=\max(0,\, 2 - \sum_{c\in\mathrm{ch}(v)} f_c)$, where the sum is over the children of~$v$. The parent must supply at least $\delta$ to satisfy $v$'s constraint.

\subsection{State and transitions}

\begin{definition}[Tree DP state]
\label{def:treestate}
For each vertex $v$, the DP is a map
\[
  \mathrm{dp}_v(f_v,\, \delta,\, w) = \text{number of valid assignments in the subtree of } v
\]
where $f_v\in\{0,1,2\}$, $\delta\in\{0,1,2\}$, and $w\in\{0,\ldots,2\,|\text{subtree}(v)|\}$.
\end{definition}

\textbf{Leaf vertex.} If $v$ is a leaf (no children):
\begin{itemize}
  \item $f_v=0$: $\delta=2$ (needs $2$ from parent), $w=0$. Count: $\mathrm{dp}_v(0,2,0)=1$.
  \item $f_v=1$: $\delta=0$ (no need from parent), $w=1$. Count: $\mathrm{dp}_v(1,0,1)=1$.
  \item $f_v=2$: $\delta=0$, $w=2$. Count: $\mathrm{dp}_v(2,0,2)=1$.
\end{itemize}

\textbf{Internal vertex.} Let $v$ have children $c_1, c_2, \ldots, c_m$. We merge children one by one. Initialize a temporary map:
\[
  \textit{temp}(f_v,\, \delta_0,\, f_v) = 1, \qquad \delta_0 = \begin{cases} 2, & f_v=0, \\ 0, & f_v>0. \end{cases}
\]
For each child $c_j$ (with DP $\mathrm{dp}_{c_j}$), update:
\[
  \textit{temp}_{\textit{new}}(f_v,\, \delta',\, w') \;\mathrel{+}=\; \textit{temp}(f_v,\, \delta,\, w) \cdot \mathrm{dp}_{c_j}(f_c,\, \delta_c,\, w_c),
\]
where:
\begin{itemize}
  \item $f_c$ is the child's value, $\delta_c$ is the child's deficit, $w_c$ is the child's subtree weight;
  \item \textbf{Child satisfaction}: if $\delta_c > 0$, then $f_v \ge \delta_c$ is required (the parent must supply the deficit);
  \item $\delta' = \max(0,\, \delta - f_c)$ if $f_v=0$ (the child's value reduces $v$'s deficit), or $\delta'=0$ if $f_v>0$;
  \item $w' = w + w_c$.
\end{itemize}
After processing all children, set $\mathrm{dp}_v = \textit{temp}$.

\textbf{Root.} At the root $r$, there is no parent. Hence, only states with $\delta=0$ are valid:
\[
  d_I(T, k) = \sum_{f_r=0}^{2} \mathrm{dp}_r(f_r,\, 0,\, k).
\]

\subsection{Algorithm}

\begin{algorithm}[htbp]
\caption{Tree Italian domination polynomial}
\label{alg:tree}
\small
\begin{algorithmic}[1]
\Require Tree $T=(V,E)$ with $n$ vertices
\Ensure Coefficients $d_I(T, k)$ for $k=0,\ldots,2n$
\State Root $T$ at vertex~$r$; compute post-order $\pi$
\ForAll{$v$ in post-order $\pi$}
  \If{$v$ is a leaf}
    \State $\mathrm{dp}_v(0, 2, 0)\gets 1$
    \State $\mathrm{dp}_v(1, 0, 1)\gets 1$
    \State $\mathrm{dp}_v(2, 0, 2)\gets 1$
  \Else
    \State Initialize $\textit{temp}$: for each $f_v\in\{0,1,2\}$, set $\textit{temp}(f_v, \delta_0, f_v)\gets 1$ with $\delta_0=2$ if $f_v=0$, else $\delta_0=0$
    \ForAll{children $c$ of $v$}
      \State Initialize $\textit{temp\_new}$ as empty
      \ForAll{$(f_v, \delta, w)$ in $\textit{temp}$ with count $>0$}
        \ForAll{$(f_c, \delta_c, w_c)$ in $\mathrm{dp}_c$ with count $>0$}
          \If{$\delta_c>0$ \textbf{and} $f_v<\delta_c$} \textbf{continue} \EndIf
          \State $\delta'\gets \max(0, \delta - f_c)$ if $f_v=0$, else $0$
          \State $w'\gets w + w_c$
          \State $t \gets \textit{temp}(f_v,\delta,w)\cdot\mathrm{dp}_c(f_c,\delta_c,w_c)$
          \State $\textit{temp\_new}(f_v, \delta', w')\gets \textit{temp\_new}(f_v, \delta', w')+t$
        \EndFor
      \EndFor
      \State $\textit{temp}\gets\textit{temp\_new}$
    \EndFor
    \State $\mathrm{dp}_v\gets\textit{temp}$
  \EndIf
\EndFor
\State Initialize $\textit{result}[k]\gets 0$ for all $k$
\ForAll{$(f_r, 0, w)$ in $\mathrm{dp}_r$ with count $>0$}
  \State $\textit{result}[w]\mathrel{+}=\mathrm{dp}_r(f_r, 0, w)$
\EndFor
\State \Return $\textit{result}$
\end{algorithmic}
\end{algorithm}

\subsection{Correctness}

\begin{theorem}
\label{thm:treecorrect}
Algorithm~\ref{alg:tree} correctly computes $d_I(T, k)$ for all $k$.
\end{theorem}

\begin{proof}
We prove by structural induction that $\mathrm{dp}_v(f_v, \delta, w)$ equals the number of valid partial assignments in the subtree of $v$ such that:
\begin{enumerate}
  \item[(I1)] Every vertex in the subtree of $v$, except possibly $v$ itself, satisfies the Italian domination constraint considering only neighbors within the subtree.
  \item[(I2)] $f(v) = f_v$.
  \item[(I3)] If $f_v = 0$, then $\delta = \max(0,\, 2 - \sum_{c\in\text{children}(v)} f(c))$, i.e., $\delta$ is the additional amount $v$ needs from its parent. If $f_v > 0$, then $\delta = 0$.
  \item[(I4)] $w = \sum_{u\in\text{subtree}(v)} f(u)$.
\end{enumerate}

\emph{Base case (leaf)}: A leaf $v$ has no children. If $f_v=0$, the only neighbor is the parent, so $\delta=2$ (needs $2$ from parent) and $w=0$. If $f_v>0$, no constraint is violated ($v$ itself is not zero), so $\delta=0$ and $w=f_v$. The invariant holds.

\emph{Inductive step (internal vertex)}: Let $v$ have children $c_1,\ldots,c_m$ with already-computed DPs. Consider the merge of child $c_j$ into the accumulated state $\textit{temp}$.

By the induction hypothesis, each state in $\mathrm{dp}_{c_j}$ satisfies (I1)--(I4) for the subtree of~$c_j$. The key points are:
\begin{itemize}
  \item \textbf{Child satisfaction}: When $\delta_c > 0$ (child $c_j$ needs $\delta_c$ from its parent $v$), we require $f_v \ge \delta_c$. This is exactly the condition that $v$'s value satisfies $c_j$'s deficit. If this holds, then $c_j$'s Italian domination constraint is fully satisfied (by its children within the subtree, plus the contribution $f_v$ from its parent).

  \item \textbf{Deficit update}: When $f_v = 0$, the deficit of $v$ is $\max(0, 2 - \sum_{c} f(c))$. Since we process children sequentially, the running deficit after child $c_j$ is $\max(0, \delta_{\text{before}} - f_{c_j})$. This is correct because $\max(0, 2 - a - b) = \max(0, \max(0, 2-a) - b)$ for $a, b\ge 0$ (proved by case analysis on the sign of $2-a-b$).

  \item \textbf{Weight accumulation}: $w' = w + w_c$ correctly accumulates the subtree weight.
\end{itemize}

After processing all children, $\mathrm{dp}_v$ satisfies the invariant.

\emph{Root}: At the root $r$, there is no parent, so the deficit must be~$0$. Any state with $\delta>0$ represents an assignment where $r$'s constraint is not fully satisfied and cannot be completed; these states are correctly excluded. The remaining states have all constraints satisfied, and summing by weight gives $d_I(T, k)$.

By structural induction, the claim follows.
\end{proof}

\subsection{Complexity}

\begin{theorem}
\label{thm:treecomplex}
Algorithm~\ref{alg:tree} runs in $O(n^2)$ time and $O(n)$ space for a tree with $n$ vertices.
\end{theorem}

\begin{proof}
The state at each vertex $v$ has $3$ choices for $f_v$, $3$ choices for $\delta$, and at most $2\,|\text{subtree}(v)|{+}1$ choices for~$w$, giving $9\cdot(2|\text{subtree}(v)|+1) = O(|\text{subtree}(v)|)$ states per vertex.

When merging child $c$ into the accumulated set at vertex~$v$, the work is proportional to $|\textit{temp}|\cdot|\mathrm{dp}_c|$. The accumulated weight range is bounded by $2$ times the number of vertices processed so far, and the child's weight range is bounded by $2|\text{subtree}(c)|$.

By the standard tree knapsack argument, each pair of vertices $(u_1, u_2)$ contributes to the merging cost at most once, namely at their lowest common ancestor. The per-pair cost is $O(1)$ (from the constant state dimension factors $3\times 3\times 3\times 3=81$), and the weight factors contribute an additional factor proportional to the subtree sizes. More precisely, the total work across all merges is:
\[
  \sum_{v\in V}\sum_{c\in\text{children}(v)} O(|\text{subtree}_{\text{acc}}|\cdot|\text{subtree}(c)|) = O(n^2),
\]
where $\text{subtree}_{\text{acc}}$ is the accumulated subtree size during the merge at~$v$. This bound follows from the fact that each pair of vertices is counted at most once across all merges (at their LCA).

Space: at any point during the post-order traversal, the live DPs belong to pairwise vertex-disjoint subtrees (the current vertex's accumulated map, plus the DPs of already-processed but not yet merged siblings along the root--current-vertex path). The sum of their sizes is therefore at most $O(n)$, giving $O(n)$ total space.
\end{proof}

\begin{remark}
For path graphs (which are trees), Algorithm~\ref{alg:tree} also applies, but Algorithm~\ref{alg:path} is more efficient in practice due to its simpler state structure (no deficit tracking needed, since the constraint is propagated linearly).
\end{remark}

\section{Algorithm for Cycle Graphs}
\label{sec:cycle}

We now develop a dynamic programming algorithm for cycle graphs $C_n$ ($n\ge 3$). The key challenge is the \emph{closing edge} $v_n v_1$, which creates a circular dependency: the constraint at $v_1$ involves $v_n$, and vice versa. We resolve this by \emph{boundary conditioning}: fixing the values of $v_1$ and $v_n$ and reducing the problem to a modified path DP. This is a standard conditioning technique, specialized here to the Italian constraint; in deficit variables it amounts to fixing the two boundary entries of the composition and running the path DP of Section~\ref{sec:path} on the remaining $n-2$ positions.

\subsection{Problem formulation}

Label the vertices of $C_n$ as $v_1, v_2, \ldots, v_n$ in cyclic order, with edges $v_i v_{i+1}$ for $i=1,\ldots,n-1$ and $v_n v_1$. An Italian dominating function $f$ corresponds to a sequence $(x_1, x_2, \ldots, x_n)$ where $x_i = f(v_i) \in \{0, 1, 2\}$, satisfying:
\begin{enumerate}
  \item[(D1)] For $i=1$: if $x_1=0$, then $x_2 + x_n \ge 2$.
  \item[(D2)] For $2\le i\le n-1$: if $x_i=0$, then $x_{i-1}+x_{i+1}\ge 2$.
  \item[(D3)] For $i=n$: if $x_n=0$, then $x_{n-1}+x_1\ge 2$.
\end{enumerate}

Note that constraints (D1) and (D3) differ from the path constraints (C1) and (C3): in the cycle, $v_1$ has two neighbors ($v_2$ and $v_n$), whereas in the path, $v_1$ has only one neighbor ($v_2$).

\subsection{Boundary conditioning}

The central idea is to \emph{condition on} the pair $(x_1, x_n) = (a, b)$ where $a, b \in \{0, 1, 2\}$. For each fixed pair, the remaining vertices $v_2, \ldots, v_{n-1}$ form a path with modified boundary constraints:

\begin{enumerate}
  \item \textbf{Left boundary} (from cycle constraint at $v_1$): If $a = 0$, then $x_2 + b \ge 2$, so $x_2 \ge \max(0, 2-b)$. If $a > 0$, no constraint on $x_2$ from~$v_1$.

  \item \textbf{Right boundary} (from cycle constraint at $v_n$): If $b = 0$, then $x_{n-1} + a \ge 2$, so $x_{n-1} \ge \max(0, 2-a)$. If $b > 0$, no constraint on $x_{n-1}$ from~$v_n$.

  \item \textbf{Middle constraints} (same as path): For $2\le i\le n-1$, if $x_i = 0$, then $x_{i-1} + x_{i+1} \ge 2$.
\end{enumerate}

We define:
\[
c_{\text{init}}(a, b) = \begin{cases} \max(0, 2-b), & \text{if } a = 0, \\ 0, & \text{if } a > 0, \end{cases}
\]
as the initial lower-bound constraint on $x_2$, and
\[
c_{\text{cycle}}(a, b) = \begin{cases} \max(0, 2-a), & \text{if } b = 0, \\ 0, & \text{if } b > 0, \end{cases}
\]
as the additional lower-bound constraint on $x_{n-1}$ from the cycle.

\subsection{DP state and transitions}

For each fixed pair $(a, b)$, we run a path DP on $v_2, \ldots, v_{n-1}$ with the following modifications:

\textbf{Initialization} (at $v_2$, with $\textit{prev} = a$):
\begin{itemize}
  \item Enumerate $x_2 \ge c_{\text{init}}(a, b)$.
  \item If $x_2 = 0$: set $c_{\text{new}} = \max(0, 2-a)$ (path constraint at $v_2$: $x_3 \ge 2 - x_1 = 2-a$).
  \item If $x_2 > 0$: set $c_{\text{new}} = 0$.
  \item Prefix sum: $s = a + x_2$.
\end{itemize}

\textbf{Middle positions} ($v_3, \ldots, v_{n-2}$): Same as the path DP; propagate the constraint $c$ based on the current and previous values.

\textbf{Terminal} (at $v_{n-1}$, with $\textit{prev} = x_{n-2}$):
\begin{itemize}
  \item Enumerate $x_{n-1} \ge \max(c_{\text{prev}},\, c_{\text{cycle}}(a, b))$.
  \item \textbf{Path constraint at $v_{n-1}$}: If $x_{n-1} = 0$, then $x_{n-2} + x_n \ge 2$, i.e., $\textit{prev} + b \ge 2$. This is equivalent to $b \ge \max(0, 2-\textit{prev})$. If this fails, skip.
  \item Weight: $k = s + x_{n-1} + b$.
\end{itemize}

The cycle polynomial is then:
\[
  D_I(C_n, x) = \sum_{a=0}^{2} \sum_{b=0}^{2} x^{a+b} \cdot Q_{a,b}(n, x),
\]
where $Q_{a,b}(n, x)$ is the generating polynomial for valid middle assignments given the boundary pair $(a, b)$.

\subsection{Algorithm}

\begin{algorithm}[htbp]
\caption{Cycle graph Italian domination polynomial}
\label{alg:cycle}
\small
\begin{algorithmic}[1]
\Require Cycle length $n\ge 3$
\Ensure Coefficients $d_I(C_n, k)$ for $k=0,\ldots,2n$
\State $\textit{max\_sum}\gets 2n$; initialize $\textit{result}[k]\gets 0$ for all $k$
\For{$a\gets 0$ \textbf{to} $2$}
  \For{$b\gets 0$ \textbf{to} $2$}
    \State $c_{\text{init}}\gets \max(0, 2-b)$ if $a=0$, else $0$
    \State $c_{\text{cycle}}\gets \max(0, 2-a)$ if $b=0$, else $0$
    \If{$n=3$}
      \Comment{Single middle vertex $v_2$}
      \For{$x_2\gets \max(c_{\text{init}}, c_{\text{cycle}})$ \textbf{to} $2$}
        \If{$x_2=0$ \textbf{and} $a+b<2$} \textbf{continue} \EndIf
        \State $k\gets a+x_2+b$;\quad $\textit{result}[k]\mathrel{+}=\!1$
      \EndFor
      \State \textbf{continue}
    \EndIf
    \State Initialize $\mathrm{dp}$ as empty map
    \For{$x_2\gets c_{\text{init}}$ \textbf{to} $2$}
      \State $c_{\text{new}}\gets \max(0,2-a)$ if $x_2=0$, else $0$
      \State $\mathrm{dp}[x_2][c_{\text{new}}][a\!+\!x_2]\mathrel{+}=\!1$
    \EndFor
    \For{$i\gets 3$ \textbf{to} $n-2$}
      \State Initialize $\mathrm{dp\_new}$ as empty map
      \ForAll{$(\textit{prev}, c_{\textit{prev}}, s)$ in $\mathrm{dp}$}
        \For{$x_i\gets c_{\textit{prev}}$ \textbf{to} $2$}
          \State $s'\gets s+x_i$
          \If{$s'>\textit{max\_sum}$} \textbf{continue} \EndIf
          \State $c_{\text{new}}\gets \max(0,2\!-\!\textit{prev})$ if $x_i=0$, else $0$
          \State $t \gets \mathrm{dp}[\textit{prev}][c_{\textit{prev}}][s]$
          \State $\mathrm{dp\_new}[x_i][c_{\text{new}}][s']\gets \mathrm{dp\_new}[x_i][c_{\text{new}}][s']+t$
        \EndFor
      \EndFor
      \State $\mathrm{dp}\gets \mathrm{dp\_new}$
    \EndFor
    \ForAll{$(\textit{prev}, c_{\textit{prev}}, s)$ in $\mathrm{dp}$}
      \For{$x_{n-1}\gets \max(c_{\textit{prev}}, c_{\text{cycle}})$ \textbf{to} $2$}
        \If{$x_{n-1}=0$ \textbf{and} $b<\max(0, 2-\textit{prev})$} \textbf{continue} \EndIf
        \State $k\gets s+x_{n-1}+b$
        \State $\textit{result}[k]\mathrel{+}=\mathrm{dp}[\textit{prev}][c_{\textit{prev}}][s]$
      \EndFor
    \EndFor
  \EndFor
\EndFor
\State \Return $\textit{result}$
\end{algorithmic}
\end{algorithm}

\subsection{Correctness}

\begin{theorem}
\label{thm:cyclecorrect}
Algorithm~\ref{alg:cycle} correctly computes $d_I(C_n, k)$ for all $k$.
\end{theorem}

\begin{proof}
The cycle $C_n$ has vertex set $\{v_1, \ldots, v_n\}$ and edge set $\{v_i v_{i+1} : 1 \le i \le n-1\} \cup \{v_n v_1\}$. We partition all Italian dominating functions on $C_n$ by the pair $(f(v_1), f(v_n)) = (a, b)$. Since the pairs are disjoint and exhaustive, it suffices to show that for each fixed $(a, b)$, the algorithm correctly counts the number of valid assignments to $v_2, \ldots, v_{n-1}$.

Fix $(a, b)$. The constraints on the middle sequence $(x_2, \ldots, x_{n-1})$ are:
\begin{enumerate}
  \item[(E1)] \emph{Cycle constraint at $v_1$}: If $a=0$, then $x_2 + b \ge 2$, i.e., $x_2 \ge \max(0, 2-b) = c_{\text{init}}$.
  \item[(E2)] \emph{Path constraint at $v_2$}: If $x_2=0$, then $a + x_3 \ge 2$, i.e., $x_3 \ge \max(0, 2-a)$. This is propagated as $c_{\text{new}} = \max(0, 2-a)$ when $x_2 = 0$.
  \item[(E3)] \emph{Path constraints at $v_i$} ($3 \le i \le n-2$): If $x_i=0$, then $x_{i-1}+x_{i+1} \ge 2$. This is the standard path constraint, propagated as $c_{\text{new}} = \max(0, 2-\textit{prev})$.
  \item[(E4)] \emph{Path constraint at $v_{n-1}$}: If $x_{n-1}=0$, then $\textit{prev} + b \ge 2$, i.e., $b \ge \max(0, 2-\textit{prev})$. This is checked at the terminal.
  \item[(E5)] \emph{Cycle constraint at $v_n$}: If $b=0$, then $x_{n-1} + a \ge 2$, i.e., $x_{n-1} \ge \max(0, 2-a) = c_{\text{cycle}}$.
\end{enumerate}

The initialization enforces (E1) by requiring $x_2 \ge c_{\text{init}}$, and correctly sets the propagated constraint for $v_3$ based on (E2). The middle loop propagates constraints (E3) exactly as in the path DP (Theorem~\ref{thm:pathcorrect}). The terminal enforces both (E4) and (E5): it requires $x_{n-1} \ge \max(c_{\text{prev}}, c_{\text{cycle}})$ (combining the propagated constraint from the path and the cycle constraint (E5)), and checks the path constraint (E4) when $x_{n-1} = 0$.

The $n=3$ special case handles the situation where $v_2$ is the only middle vertex: both $c_{\text{init}}$ and $c_{\text{cycle}}$ apply to $x_2$, and the path constraint at $v_2$ becomes $a + b \ge 2$ (since both neighbors are fixed).

By the same induction argument as Theorem~\ref{thm:pathcorrect}, the DP correctly counts valid partial assignments at each position. Summing over all $9$ pairs $(a, b)$ gives $d_I(C_n, k)$.
\end{proof}

\subsection{Complexity}

\begin{theorem}
\label{thm:cyclecomplex}
Algorithm~\ref{alg:cycle} runs in $O(n^2)$ time and $O(n)$ space.
\end{theorem}

\begin{proof}
The outer loop iterates over $9$ pairs $(a, b)$. For each pair, the inner path DP on $n-2$ vertices has the same structure as Algorithm~\ref{alg:path}: the state space is $O(n)$ (3 choices for $\textit{prev}$, 3 for $c$, $O(n)$ for $s$), and each state transitions to at most $3$ next states. The total work per pair is $O(n^2)$, giving $O(9n^2) = O(n^2)$ total.

Space: the path DP maintains two maps of size $O(n)$ per pair, and the outer loop is sequential, so the total space is $O(n)$.
\end{proof}

\section{Known Values, Dominance, and Closed Forms for Cycles}
\label{sec:known}

\subsection{Italian domination number of cycles}

As stated in the Preliminaries (Proposition~\ref{prop:gamma_cycle}), $\gamma_I(C_n)=\lceil n/2\rceil$ is a known result, recorded in \cite{ferrari2025}; we do not reproduce its proof here.

\subsection{Relationship between cycle and path polynomials}

\begin{remark}
\label{rem:dom}
For all $n \ge 3$ and all $k \ge 0$, $d_I(C_n, k) \ge d_I(P_n, k)$, i.e., $D_I(C_n, x) \ge D_I(P_n, x)$ coefficient-wise. Indeed, every IDF on $P_n$ is also an IDF on $C_n$, because the cycle is obtained from the path by adding the edge $v_n v_1$, which only enlarges the neighborhoods of $v_1$ and $v_n$; since the Italian condition constrains only the sum of $f$ over the neighborhood, enlarging a neighborhood can never invalidate a previously valid function. Equivalently, in the deficit formulation of Section~\ref{sec:partitions}, the cycle has no endpoint constraints of the form $y_1=2\Rightarrow y_2=0$ (Remark~\ref{rem:endpoints}), so every admissible path composition is admissible on the cycle. The numerical tables in Section~\ref{sec:numerics} confirm the inequality for all tested values.
\end{remark}

\subsection{Closed-form expressions for high weights}
\label{sec:closedform}

We derive closed-form expressions for $d_I(C_n, 2n - j)$ for small values of $j$. By Corollary~\ref{cor:cycle_deficit}, $d_I(C_n,2n-j)$ equals the number of vectors $y\in\{0,1,2\}^n$ with $\sum_i y_i=j$ such that $y_i=2$ implies $y_{i-1}+y_{i+1}\le 2$. By the exclusion identity \eqref{eq:tme}, this equals $T(n,j)-E(n,j)$, where $T(n,j)=[z^j](1+z+z^2)^n$ is the trinomial coefficient of Definition~\ref{def:trinom} and $E(n,j)$ counts the excluded configurations; we enumerate $T(n,j)$ and $E(n,j)$ for $j\le 5$.

\begin{theorem}
\label{thm:closedform}
For $n \ge 3$:
\begin{enumerate}
\item[(i)] $d_I(C_n, 2n) = 1$.
\item[(ii)] $d_I(C_n, 2n-1) = n$.
\item[(iii)] $d_I(C_n, 2n-2) = \dfrac{n(n+1)}{2}$.
\item[(iv)] $d_I(C_n, 2n-3) = \dfrac{n(n-1)(n+4)}{6}$.
\item[(v)] $d_I(C_n, 2n-4) = \dfrac{n(n-1)(n^2+7n-6)}{24}$.
\end{enumerate}
For $n \ge 4$:
\begin{enumerate}
\item[(vi)] $d_I(C_n, 2n-5) = \dfrac{n(n^4+10n^3-25n^2-10n-216)}{120}$.
\end{enumerate}
\end{theorem}

\begin{proof}
Items (i)--(iii) also follow from Proposition~\ref{prop:baseline} applied to $C_n$ (which has minimum degree $2$ for $n\ge 3$); the exclusion identity \eqref{eq:tme} organizes the derivations for all $j$. Throughout, indices are taken modulo $n$, and we count admissible deficit vectors $y$ (Corollary~\ref{cor:cycle_deficit}).

\smallskip
\noindent\textit{(i) $j=0$.} $T(n,0)=1$ (all entries $0$), and $E(n,0)=0$.

\smallskip
\noindent\textit{(ii) $j=1$.} $T(n,1)=n$ (a single entry equal to $2$). Such a vector has $y$-neighbors all equal to $0$, so the constraint $y_{i-1}+y_{i+1}\le 2$ holds; $E(n,1)=0$.

\smallskip
\noindent\textit{(iii) $j=2$.} $T(n,2)=\binom{n}{2}+n$ (two entries equal to $1$, or one entry equal to $2$). No constraint is violated: a $2$-entry has neighbors summing to $0$, and $1$-entries are unconstrained. Hence $E(n,2)=0$.

\smallskip
\noindent\textit{(iv) $j=3$.} $T(n,3)=n(n-1)+\binom{n}{3}$ (one $2$-entry and one $1$-entry, or three $1$-entries). A $2$-entry has neighbors that are $0$ or $1$, summing to at most $1\le 2$; hence $E(n,3)=0$. So
\[
  d_I(C_n,2n-3)=n(n-1)+\binom{n}{3}=\frac{n(n-1)(n+4)}{6}.
\]

\smallskip
\noindent\textit{(v) $j=4$.} $T(n,4)=\binom{n}{2}+n\binom{n-1}{2}+\binom{n}{4}$ (two $2$-entries; one $2$-entry and two $1$-entries; four $1$-entries). A $2$-entry has neighbors summing to at most $2$: if the other $2$-entry is adjacent, the neighbors of each $2$-entry are $\{2,0\}$, summing to $2\le 2$; otherwise they sum to at most $1$. Hence $E(n,4)=0$, and
\[
  d_I(C_n,2n-4)=\binom{n}{2}+n\binom{n-1}{2}+\binom{n}{4}=\frac{n(n-1)(n^2+7n-6)}{24}.
\]

\smallskip
\noindent\textit{(vi) $j=5$ ($n\ge 4$).} $T(n,5)=\binom{n}{2}(n-2)+n\binom{n-1}{3}+\binom{n}{5}$ (two $2$-entries and one $1$-entry; one $2$-entry and three $1$-entries; five $1$-entries). Only the first family can violate the constraint. A vector with two $2$-entries $y_i=y_{i+1}=2$ and one $1$-entry is invalid exactly when the $1$-entry lies at $i-1$ or $i+2$, since then one of the $2$-entries has neighbor sum $1+2=3>2$; these are precisely the configurations where a $2$-entry is adjacent to both a $2$-entry and a $1$-entry. There are $n$ choices for the adjacent pair $\{i,i+1\}$ and $2$ choices for the position of the $1$-entry, so $E(n,5)=2n$ (cf.\ \eqref{eq:e5}). Hence
\[
  d_I(C_n,2n-5)
  \;=\; \binom{n}{2}(n-2)+n\binom{n-1}{3}+\binom{n}{5}-2n .
\]
Expanding and collecting over the common denominator $120$ gives
\[
  d_I(C_n,2n-5)=\frac{n(n^4+10n^3-25n^2-10n-216)}{120}.
\]
(For $n=4$, the formula gives $8$, matching Table~\ref{tab:cycle_a}; for $n=3$, $2n-5=1<\gamma_I(C_3)=2$, so $d_I(C_3,1)=0$ and the formula is not applicable.)
\end{proof}

\begin{remark}
The formulas in Theorem~\ref{thm:closedform} have been verified computationally for $n = 3, \ldots, 12$ (items (i)--(v)) and $n = 4, \ldots, 12$ (item (vi)) by brute-force enumeration.
\end{remark}

\begin{remark}
The derivation above highlights the connection with constrained integer compositions: for $j\le 4$ the cycle constraint never binds, so $d_I(C_n,2n-j)$ equals the trinomial coefficient $T(n,j)$; for $j=5$ it equals $T(n,5)$ minus $2n$ explicitly excluded configurations. By Lemma~\ref{lem:poly}, the polynomiality of $d_I(C_n,2n-j)$ in $n$ (of degree $j$, with leading coefficient $1/j!$) is a theorem rather than a conjecture; the exclusion-counting method of this section gives the systematic route to explicit closed forms for all fixed $j$.
\end{remark}

\begin{remark}[Comparison with the Roman domination polynomial of cycles]
\label{rem:romancompare}
Deepak, Alwardi and Indiramma \cite{deepak2021} derived closed forms for the coefficients $r(C_n,k)$ of the Roman domination polynomial of cycles. Their values $r(C_n,2n-2)=n(n+1)/2$ and $r(C_n,2n-3)=n(n-1)(n+4)/6$ coincide with items (iii) and (iv) of Theorem~\ref{thm:closedform}, whereas for $j=4$ and $j=5$ the formulas differ. The coincidence is a weight-conservation phenomenon: for $j\le 3$ a vertex with $f(v)=0$ (deficit $2$) cannot have two neighbors with $f=1$ (deficits $1$ and $1$), since the total deficit would be at least $4$; hence no configuration exploits the relaxation that distinguishes Italian from Roman domination, and the two polynomials agree for $j\le 3$. For $j=4$ such configurations first appear, e.g., the deficit pattern $(1,2,1)$ on three consecutive positions (i.e., $f=(1,0,1)$ on three consecutive vertices), which is a valid Italian dominating function but not a Roman dominating function; this explains the divergence of the closed forms for $j\ge 4$.
\end{remark}

\section{Numerical Results}
\label{sec:numerics}

All coefficients in the tables below were computed with the corresponding dynamic programming algorithms (Sections~\ref{sec:path}--\ref{sec:cycle}) and independently verified by brute-force enumeration for every tabulated value. The verifier enumerates all $3^n$ candidate functions (for $n\le 12$) and checks the Italian condition directly; it is a compact Python script that will accompany the arXiv version of this paper.

\subsection{Italian domination on path graphs}

Table~\ref{tab:path_a} and Table~\ref{tab:path_b} present the coefficients $d_I(P_n, k)$ for $n=1,\ldots,10$ and $k=1,\ldots,20$, computed by Algorithm~\ref{alg:path}.

\begin{table}[ht]
\centering
\caption{Coefficients $d_I(P_n, k)$ for path graphs, $k=1,\ldots,10$}
\label{tab:path_a}
\begin{tabular}{c|*{10}{r}}
\toprule
$n\backslash k$ & 1 & 2 & 3 & 4 & 5 & 6 & 7 & 8 & 9 & 10 \\
\midrule
1 & 1 & 1 & --- & --- & --- & --- & --- & --- & --- & --- \\
2 & 0 & 3 & 2 & 1 & --- & --- & --- & --- & --- & --- \\
3 & 0 & 2 & 5 & 6 & 3 & 1 & --- & --- & --- & --- \\
4 & 0 & 0 & 4 & 13 & 14 & 10 & 4 & 1 & --- & --- \\
5 & 0 & 0 & 1 & 13 & 29 & 37 & 28 & 15 & 5 & 1 \\
6 & 0 & 0 & 0 & 6 & 32 & 73 & 92 & 80 & 48 & 21 \\
7 & 0 & 0 & 0 & 1 & 19 & 86 & 178 & 237 & 218 & 149 \\
8 & 0 & 0 & 0 & 0 & 6 & 62 & 220 & 449 & 604 & 594 \\
9 & 0 & 0 & 0 & 0 & 1 & 27 & 179 & 576 & 1129 & 1556 \\
10 & 0 & 0 & 0 & 0 & 0 & 7 & 96 & 513 & 1492 & 2875 \\
\bottomrule
\end{tabular}
\end{table}

\begin{table}[ht]
\centering
\caption{Coefficients $d_I(P_n, k)$ for path graphs, $k=11,\ldots,20$}
\label{tab:path_b}
\begin{tabular}{c|*{10}{r}}
\toprule
$n\backslash k$ & 11 & 12 & 13 & 14 & 15 & 16 & 17 & 18 & 19 & 20 \\
\midrule
6 & 6 & 1 & --- & --- & --- & --- & --- & --- & --- & --- \\
7 & 75 & 28 & 7 & 1 & --- & --- & --- & --- & --- & --- \\
8 & 440 & 252 & 110 & 36 & 8 & 1 & --- & --- & --- & --- \\
9 & 1597 & 1274 & 800 & 398 & 154 & 45 & 9 & 1 & --- & --- \\
10 & 4006 & 4283 & 3616 & 2458 & 1350 & 597 & 208 & 55 & 10 & 1 \\
\bottomrule
\end{tabular}
\end{table}

\noindent\textbf{Note:} The entry ``---'' indicates that the corresponding weight exceeds $2n$ (the maximum possible weight), so the coefficient is~$0$.

\subsection{Italian domination on star graphs}

Table~\ref{tab:star} presents $d_I(K_{1,m}, k)$ for star graphs with $m=1,\ldots,8$ leaves, computed by Algorithm~\ref{alg:tree}.

\begin{table}[ht]
\centering
\caption{Coefficients $d_I(K_{1,m}, k)$ for star graphs}
\label{tab:star}
\begin{tabular}{c|*{12}{r}}
\toprule
$m\backslash k$ & 1 & 2 & 3 & 4 & 5 & 6 & 7 & 8 & 9 & 10 & 11 & 12 \\
\midrule
1 & 0 & 3 & 2 & 1 & --- & --- & --- & --- & --- & --- & --- & --- \\
2 & 0 & 2 & 5 & 6 & 3 & 1 & --- & --- & --- & --- & --- & --- \\
3 & 0 & 1 & 4 & 10 & 13 & 10 & 4 & 1 & --- & --- & --- & --- \\
4 & 0 & 1 & 4 & 11 & 21 & 29 & 26 & 15 & 5 & 1 & --- & --- \\
5 & 0 & 1 & 5 & 15 & 31 & 51 & 66 & 65 & 45 & 21 & 6 & 1 \\
6 & 0 & 1 & 6 & 21 & 50 & 91 & 133 & 162 & 161 & 125 & 71 & 28 \\
7 & 0 & 1 & 7 & 28 & 77 & 161 & 267 & 365 & 421 & 413 & 336 & 217 \\
8 & 0 & 1 & 8 & 36 & 112 & 266 & 504 & 785 & 1025 & 1143 & 1100 & 910 \\
\midrule
$m\backslash k$ & 13 & 14 & 15 & 16 & 17 & 18 & \multicolumn{6}{c}{} \\
\midrule
6 & 7 & 1 & --- & --- & --- & --- & \multicolumn{6}{c}{} \\
7 & 105 & 36 & 8 & 1 & --- & --- & \multicolumn{6}{c}{} \\
8 & 630 & 350 & 148 & 45 & 9 & 1 & \multicolumn{6}{c}{} \\
\bottomrule
\end{tabular}
\end{table}

\subsection{Italian domination on cycle graphs}

Table~\ref{tab:cycle_a} and Table~\ref{tab:cycle_b} present $d_I(C_n, k)$ for $n=3,\ldots,12$ and $k=1,\ldots,24$, computed by Algorithm~\ref{alg:cycle}.

\begin{table}[ht]
\centering
\small
\caption{Coefficients $d_I(C_n, k)$ for cycle graphs, $k=1,\ldots,12$}
\label{tab:cycle_a}
\begin{tabular}{c|*{12}{r}}
\toprule
$n\backslash k$ & 1 & 2 & 3 & 4 & 5 & 6 & 7 & 8 & 9 & 10 & 11 & 12 \\
\midrule
3 & 0 & 6 & 7 & 6 & 3 & 1 & --- & --- & --- & --- & --- & --- \\
4 & 0 & 2 & 8 & 19 & 16 & 10 & 4 & 1 & --- & --- & --- & --- \\
5 & 0 & 0 & 5 & 25 & 41 & 45 & 30 & 15 & 5 & 1 & --- & --- \\
6 & 0 & 0 & 2 & 18 & 60 & 105 & 114 & 90 & 50 & 21 & 6 & 1 \\
7 & 0 & 0 & 0 & 7 & 49 & 154 & 260 & 301 & 252 & 161 & 77 & 28 \\
8 & 0 & 0 & 0 & 2 & 24 & 144 & 384 & 659 & 784 & 704 & 488 & 266 \\
9 & 0 & 0 & 0 & 0 & 9 & 87 & 387 & 990 & 1663 & 2052 & 1935 & 1446 \\
10 & 0 & 0 & 0 & 0 & 2 & 35 & 270 & 1055 & 2530 & 4253 & 5350 & 5285 \\
11 & 0 & 0 & 0 & 0 & 0 & 11 & 132 & 814 & 2827 & 6523 & 10891 & 13992 \\
12 & 0 & 0 & 0 & 0 & 0 & 2 & 48 & 462 & 2360 & 7572 & 16836 & 28033 \\
\bottomrule
\end{tabular}
\end{table}

\begin{table}[ht]
\centering
\small
\caption{Coefficients $d_I(C_n, k)$ for cycle graphs, $k=13,\ldots,24$}
\label{tab:cycle_b}
\begin{tabular}{c|*{12}{r}}
\toprule
$n\backslash k$ & 13 & 14 & 15 & 16 & 17 & 18 & 19 & 20 & 21 & 22 & 23 & 24 \\
\midrule
7 & 7 & 1 & --- & --- & --- & --- & --- & --- & --- & --- & --- & --- \\
8 & 112 & 36 & 8 & 1 & --- & --- & --- & --- & --- & --- & --- & --- \\
9 & 864 & 414 & 156 & 45 & 9 & 1 & --- & --- & --- & --- & --- & --- \\
10 & 4190 & 2710 & 1432 & 615 & 210 & 55 & 10 & 1 & --- & --- & --- & --- \\
11 & 14322 & 11968 & 8261 & 4741 & 2255 & 880 & 275 & 66 & 11 & 1 & --- & --- \\
12 & 36588 & 38670 & 33760 & 24666 & 15168 & 7858 & 3408 & 1221 & 352 & 78 & 12 & 1 \\
\bottomrule
\end{tabular}
\end{table}

\subsection{Examples on specific trees}

We verify Algorithm~\ref{alg:tree} on several non-path, non-star trees.

\begin{example}[Complete binary tree with 7 vertices]
\label{ex:binary}
Consider the complete binary tree with root~$v_1$, children $v_2, v_3$, and grandchildren $v_4, v_5$ (under~$v_2$) and $v_6, v_7$ (under~$v_3$). Algorithm~\ref{alg:tree} gives:
\[
\begin{aligned}
  D_I(T, x) &= 3x^4 + 18x^5 + 59x^6 + 125x^7 + 182x^8 + 186x^9 \\
  &\quad + 137x^{10} + 73x^{11} + 28x^{12} + 7x^{13} + x^{14},
\end{aligned}
\]
verified by brute force.
\end{example}

\begin{example}[Spider $S(2,2,2)$]
\label{ex:spider}
The spider $S(2,2,2)$ is the tree obtained from the star $K_{1,3}$ by subdividing each edge once, giving 7 vertices: a center connected to three paths of length~$2$. Algorithm~\ref{alg:tree} gives:
\[
  D_I(T, x) = x^4 + 10x^5 + 65x^6 + 150x^7 + 213x^8 + 203x^9 + 143x^{10} + 74x^{11} + 28x^{12} + 7x^{13} + x^{14},
\]
verified by brute force.
\end{example}

\begin{example}[Caterpillar tree]
\label{ex:caterpillar}
Consider the caterpillar tree with backbone $v_1$--$v_2$--$v_3$--$v_4$ and pendant leaves $v_5$ at~$v_2$ and $v_6$ at~$v_3$ (6 vertices total). Algorithm~\ref{alg:tree} gives:
\[
  D_I(T, x) = 4x^4 + 20x^5 + 49x^6 + 72x^7 + 70x^8 + 46x^9 + 21x^{10} + 6x^{11} + x^{12},
\]
verified by brute force.
\end{example}

\begin{example}[Italian domination polynomials for small cycles]
\label{ex:cyclepoly}
The Italian domination polynomials for small cycles are:
\begin{align*}
  D_I(C_3, x) &= 6x^2 + 7x^3 + 6x^4 + 3x^5 + x^6, \\
  D_I(C_4, x) &= 2x^2 + 8x^3 + 19x^4 + 16x^5 + 10x^6 + 4x^7 + x^8, \\
  D_I(C_5, x) &= 5x^3 + 25x^4 + 41x^5 + 45x^6 + 30x^7 + 15x^8 + 5x^9 + x^{10}, \\
  D_I(C_6, x) &= 2x^3 + 18x^4 + 60x^5 + 105x^6 + 114x^7 + 90x^8 + 50x^9 \\
  &\quad + 21x^{10} + 6x^{11} + x^{12}.
\end{align*}
All verified by brute force.
\end{example}

\subsection{Observations}

From the numerical tables, we observe the following patterns:

\begin{enumerate}
  \item \textbf{Maximum weight}: $d_I(G, 2n)=1$ for all graphs with $n$ vertices, corresponding to $f\equiv 2$ (Proposition~\ref{prop:basic}).

  \item \textbf{Minimum weight on paths}: $\gamma_I(P_n) = \lfloor n/2\rfloor + 1$ (Proposition~\ref{prop:gamma_path}), giving the sequence $1, 2, 2, 3, 3, 4, 4, 5, 5, 6, \ldots$ for $n=1,2,\ldots$

  \item \textbf{Minimum weight on cycles}: $\gamma_I(C_n) = \lceil n/2\rceil$ (Proposition~\ref{prop:gamma_cycle}), giving the sequence $2, 2, 3, 3, 4, 4, 5, 5, 6, 6, \ldots$ for $n=3,4,\ldots$

  \item \textbf{Cycle dominates path}: The dominance $d_I(C_n, k) \ge d_I(P_n, k)$ (Remark~\ref{rem:dom}) is confirmed for all values tested.

  \item \textbf{Unimodality (conjecture)}: Both sequences $(d_I(P_n,k))_k$ and $(d_I(C_n,k))_k$ appear unimodal in every case tested.

  \item \textbf{Closed-form verification}: The formulas of Theorem~\ref{thm:closedform} match the computed values exactly for all tested $n$.

  \item \textbf{Total count on paths}: The total number of Italian dominating functions on $P_n$ grows as $2, 6, 17, 46, 129, 359, 999, 2782, 7746, 21567$ for $n=1,\ldots,10$.

  \item \textbf{Total count on cycles}: For $n=3,\ldots,12$, the total number of Italian dominating functions on $C_n$ is $23, 60, 167, 467, 1297, 3612,$ $10058, 28003, 77970,$ and $217095$. The ratio $T(n+1)/T(n)$ converges to approximately $2.7843$, the same limiting ratio as for paths.

  \item \textbf{Composition viewpoint}: In the deficit formulation (Section~\ref{sec:partitions}), the unconstrained total $\sum_j T(n,j)=3^n$ counts all compositions; the Italian constraints reduce the counts to $\approx (2.7843)^n$ on paths and cycles. The difference is accounted for entirely by the local constraints (P2). Moreover, for near-maximum weights the coefficients of $C_n$ equal the trinomial coefficients $T(n,j)$ for $j\le 4$ (Theorem~\ref{thm:closedform}).

  \item \textbf{Star graph structure}: For star graphs $K_{1,m}$, we observe $d_I(K_{1,m}, 2) = 1$ for $m\ge 3$ (only the function assigning $f(\text{center})=2$ and $f(\text{leaves})=0$ achieves weight~$2$), while $d_I(K_{1,1}, 2)=3$ and $d_I(K_{1,2}, 2)=2$.
\end{enumerate}

\section{Conclusion}

We have studied the Italian domination polynomial $D_I(G,x)$, the counting analogue of the Italian domination number, on paths, trees, and cycles. The central structural insight of the paper is that counting weight-$k$ Italian dominating sets is a \emph{constrained integer composition} problem: via the deficit transform $y(v)=2-f(v)$, the coefficient $d_I(G,2n-j)$ counts compositions of $j$ into $n$ parts, each at most~$2$, that satisfy a local capacity constraint (Lemma~\ref{lem:deficit}). This viewpoint gives clean derivations of the closed forms $d_I(C_n,2n-j)$, $j=0,\ldots,5$, as trinomial coefficients minus explicitly enumerated excluded configurations, and explains the coefficient-wise dominance of cycles over paths through endpoint constraints (Remark~\ref{rem:endpoints}).

On the algorithmic side, we presented three dynamic programming algorithms with full correctness proofs and complexity analyses: a chain DP for paths, a post-order merge DP with a deficit parameter for trees, and a boundary-conditioning DP for cycles, each running in $O(n^2)$ time and $O(n)$ space. All numerical results were verified by brute-force enumeration, the known values $\gamma_I(P_n)=\lfloor n/2\rfloor+1$ and $\gamma_I(C_n)=\lceil n/2\rceil$ were recalled with references, and Lemma~\ref{lem:poly} establishes that $d_I(C_n,2n-j)$ is a polynomial in $n$ of degree $j$ for each fixed $j$.

Several directions remain open for future work:
\begin{itemize}
  \item Extending the tree and cycle DP to broader graph classes, such as bounded-treewidth graphs and cactus graphs, where a similar state structure may apply.
  \item Deriving a complete linear recurrence for the total count $\sum_k d_I(C_n, k)$ using the transfer matrix formulation (the observed limiting ratio $\approx 2.7843$ is the dominant eigenvalue of the corresponding transfer matrix).
  \item Proving the unimodality conjecture for the coefficient sequences.
  \item Deriving explicit closed forms for $d_I(C_n,2n-j)$ for general $j$ via the exclusion-counting method of Theorem~\ref{thm:closedform} (the polynomiality of these coefficients in $n$ is already established in Lemma~\ref{lem:poly}).
  \item Exploiting the constrained-composition viewpoint further, e.g., by relating $D_I$ to $q$-analogues of the trinomial coefficients and to the Roman domination polynomial of \cite{mirhoseini2024,deepak2021}.
  \item Characterizing which non-isomorphic graphs share the same Italian domination polynomial.
\end{itemize}



\begin{thebibliography}{20}\bibitem{cockayne2004}
E.~J. Cockayne, P.~A. Dreyer~Jr., S.~M. Hedetniemi, and S.~T. Hedetniemi,
Roman domination in graphs,
\textit{Discrete Mathematics}, \textbf{278}(1--2):11--22, 2004.

\bibitem{chellali2016}
M.~Chellali, T.~W. Haynes, S.~T. Hedetniemi, and A.~A. McRae,
Roman $\{2\}$-domination,
\textit{Discrete Applied Mathematics}, \textbf{204}:22--28, 2016.

\bibitem{chellali2021}
M.~Chellali, N.~Jafari Rad, S.~M. Sheikholeslami, and L.~Volkmann,
Variants of Roman domination in graphs,
in: T.~W. Haynes, S.~T. Hedetniemi, and M.~A. Henning (Eds.),
\textit{Structures of Domination in Graphs}, Developments in Mathematics 66, Springer, Cham, 2021.

\bibitem{brouwer2012}
A.~E. Brouwer, P.~Csorba, and L.~Strijov,
The Italian domination number of a graph,
\textit{manuscript}, 2012.

\bibitem{henningklostermeyer2017}
M.~A. Henning and W.~F. Klostermeyer,
Italian domination in trees,
\textit{Discrete Applied Mathematics}, \textbf{217}:557--564, 2017.

\bibitem{klostermeyer2019}
W.~F. Klostermeyer and G.~MacGillivray,
Roman, Italian, and 2-domination,
\textit{Journal of Combinatorial Mathematics and Combinatorial Computing}, \textbf{108}:125--146, 2019.

\bibitem{haynes2020}
T.~W. Haynes, S.~T. Hedetniemi, and M.~A. Henning,
\textit{Topics in Domination in Graphs},
Developments in Mathematics 64, Springer, Cham, 2020.

\bibitem{alikhani2009}
S.~Alikhani and Y.~H. Peng,
Introduction to domination polynomial of a graph,
arXiv:0905.2251 [math.CO], 2009.

\bibitem{kotek2012}
T.~Kotek, J.~Preen, F.~Simon, P.~Tittmann, and M.~Trinks,
Recurrence relations and splitting formulas for the domination polynomial,
\textit{Electronic Journal of Combinatorics}, \textbf{19}(3):\#P47, 2012.

\bibitem{kotek2014}
T.~Kotek, J.~Preen, and P.~Tittmann,
Subset-sum representations of domination polynomials,
\textit{Graphs and Combinatorics}, \textbf{30}(3):647--660, 2014.

\bibitem{mirhoseini2024}
S.~H. Mirhoseini and N.~Jafari Rad,
On the Roman domination polynomials,
\textit{Communications in Combinatorics and Optimization}, \textbf{9}(3):595--605, 2024.

\bibitem{deepak2021}
G.~Deepak, A.~Alwardi, and M.~H. Indiramma,
Roman domination polynomial of cycles,
\textit{International Journal of Mathematical Combinatorics}, \textbf{3}:86--97, 2021.

\bibitem{zheng2014}
S.~Zheng and K.~M. Koh,
Counting the number of minimum Roman dominating functions of a graph,
arXiv:1403.1019 [cs.DS], 2014.

\bibitem{abukhzam2024}
F.~N. Abu-Khzam, H.~Fernau, and K.~Mann,
Minimal Roman dominating functions: extensions and enumeration,
\textit{Algorithmica}, \textbf{86}(6):1862--1887, 2024.

\bibitem{kijima2011}
S.~Kijima, Y.~Okamoto, and T.~Uno,
Dominating set counting in graph classes,
in: \textit{Proc. COCOON 2011}, Lecture Notes in Computer Science 6842, Springer, 2011, pp.~13--24.

\bibitem{weilu2019}
D.~Wei and C.~Lu,
Independent double Roman domination on block graphs,
arXiv:1908.00784 [math.CO], 2019.

\bibitem{mertens2024}
S.~Mertens,
Domination polynomials of the grid, the cylinder, the torus, and the king graph,
arXiv:2408.08053 [math.CO], 2024.

\bibitem{andrews1998}
G.~E. Andrews,
\textit{The Theory of Partitions},
Encyclopedia of Mathematics and its Applications 2, Cambridge University Press, Cambridge, 1998.

\bibitem{bendercanfield2009}
E.~A. Bender and E.~R. Canfield,
Locally restricted compositions II: General restrictions and infinite matrices,
\textit{Electron.~J.~Combin.}, \textbf{16}(1):\#R108, 2009.

\bibitem{ferrari2025}
A.~J. Ferrari, V.~Leoni, and M.~I. Lopez Pujato,
Characterizations of Italian graphs and Sicilian graphs,
\textit{RAIRO-Operations Research}, \textbf{59}(1):239--249, 2025.


\end{thebibliography}
\end{document}